\documentclass[12pt,reqno]{amsart}
\usepackage{fullpage}

\usepackage{ifpdf}
\ifpdf
   \usepackage{times}
\fi

\usepackage[T1]{fontenc}
\usepackage{mathrsfs}
\usepackage{epsfig}
\usepackage{color}

\newtheorem{theorem}{Theorem} [section]
\newtheorem{lemma}[theorem]{Lemma}

\newtheorem{prop}[theorem]{Proposition}
\newtheorem{cor}[theorem]{Corollary}
\theoremstyle{definition}

\newtheorem{remark}[theorem]{Remark}
\numberwithin{equation}{section}

\renewcommand{\mod}[1]{{\ifmmode\text{\rm\ (mod~$#1$)}\else\discretionary{}{}{\hbox{ }}\rm(mod~$#1$)\fi}}

\newsymbol\nmid 232D

\newcommand{\A}{{\mathcal A}}
\newcommand{\B}{{\mathcal B}}
\newcommand{\C}{{\mathcal C}}

\newcommand{\e}{{\rm e}}
\newcommand{\D}{{\mathcal D}}

\newcommand{\Z}{{\mathbb Z}}

\begin{document}

\title{On sets of integers which are both sum-free and product-free}

\author[P. Kurlberg]{P\"ar Kurlberg}
\address{Department of Mathematics\\
KTH\\
SE-10044, Stockholm, Sweden}
\email{kurlberg@math.kth.se}

\author[J. C. Lagarias] {Jeffrey C.  Lagarias}
\address{Department of Mathematics\\
University of Michigan\\
Ann Arbor, MI 48109, USA}
\email{lagarias@umich.edu}

\author[C. Pomerance]{Carl Pomerance}
\address{Mathematics Department\\
Dartmouth College\\
Hanover, NH 03755, USA}
\email{carl.pomerance@dartmouth.edu}

\date{Dec. 16, 2011}



\begin{abstract}
We consider sets of positive integers containing no sum of two elements in the set and
also no product of two elements.  We show that the upper density of such a set is
strictly smaller than $\frac12$ and that this is best possible.
Further, we also find the maximal
order for the density of such sets that are also periodic modulo some positive integer.
\end{abstract}

\maketitle

%
%
\section{Introduction}
\label{sec:introduction}

The sum-product problem in combinatorial number theory asserts that if $\A$ is a finite set of
positive integers, then either $\A+\A$ or $\A\cdot\A$ is a much larger set than $\A$,
where $\A+\B$ is the set of
sums $a+b$ with $a\in\A,b\in\B$ and $\A\cdot\B$ is the set of products $ab$ with $a\in\A,b\in\B$.
A famous conjecture by Erd\H os and Szemer\'edi~\cite{ES} asserts 
that if $\epsilon>0$
is arbitrary and $\A$ is a set of $N$ positive integers, then for $N$ sufficiently large depending
on the choice of $\epsilon$, we have
$$
|\A+\A|+|\A\cdot\A|\ge N^{2-\epsilon}.
$$
This conjecture is motivated by the cases when either $|\A+\A|$ or $|\A\cdot\A|$ is
unusually small.  For example, if $\A=\{1,2,\dots,N\}$, then $\A+\A$ is small,
namely, $|\A+\A|<2N$.  However, $\A\cdot\A$ is large since there
is some $c>0$ such that $|\A\cdot\A|>N^2/(\log N)^c$.  And if $\A=\{1,2,4,\dots,2^{N-1}\}$,
then $|\A\cdot\A|<2N$, but $|\A+\A|>N^2/2$.  The best that we currently know towards this
conjecture is that it holds with exponent $4/3$ in the place of 2, a result of Solymosi~\cite{S}.
(In fact, Solymosi proves this when $\A$ is a set of positive real numbers.)

In this paper we consider a somewhat different question:  how dense can $\A$ be if
both $\A+\A$ and $\A\cdot\A$ have no elements in common with $\A$?  If $\A\cap(\A+\A)=\emptyset$
we say that $\A$ is sum-free and if $\A\cap(\A\cdot\A)=\emptyset$ we say
$\A$ is product-free.  
Before stating the main results, we give some background on sets that
are either sum-free or product-free.

If $a\in\A$ and $\A$ is sum-free, then $\{a\}+\A$ is disjoint from $\A$, and so we immediately
have that the upper asymptotic density of $\A$ is at most $\frac12$.  Density $\frac{1}{2}$ can be achieved
by taking $\A$ as the set of odd natural numbers.  Similarly, if $\A$ is a set of residues
modulo~$n$ and is sum-free, then $D(\A):=|\A|/n$ is at most $\frac12$, and this can be
achieved when $n$ is even and $\A$ consists of the odd residues.  The maximal density for
$D(\A)$ for $\A$ a sum-free set in $\Z/n\Z$ was considered in
\cite{DY69}. 
In particular, the maximum for $D(\A)$ is $\frac13-\frac1{3n}$ if $n$ is
divisible solely 
by primes that are 1~modulo~3, it is $\frac13+\frac1{3p}$ if $n$ is divisible by some prime
that is 2~modulo~3 and $p$ is the least such, and it is $\frac13$
otherwise.  Consequently,
we have $D(\A)\le\frac25$ if $\A$ is a sum-free set in $\Z/n\Z$ and
$n$ is odd.  
It is worth noting that maximal densities of subsets of arbitrary
finite abelian groups are determined in \cite{GR05}.   For
generalizations to subsets of finite non-abelian groups, see
\cite{Kedlaya}.

The problem of the maximum density of product-free sets of positive integers,
or of subsets of $\Z/n\Z$, only recently received attention.
For subsets of the positive integers, it was shown 
in \cite{KLP11} 
that the upper density of a
product-free set must be strictly less 
than~1.  Let $D(n)$ denote the maximum value of $D(\A)$ as $\A$ runs over product-free sets in
$\Z/n\Z$.  In \cite{PS11} it was shown that $D(n)<\frac12$ for the vast majority of integers,
namely for every integer not divisible by the square of a product of 6 distinct primes.
Moreover, the density of integers which are divisible by the square of a product
of 6 distinct primes was shown to be smaller than $1.56\times10^{-8}$.  

Somewhat surprisingly, $D(n)$ can in fact be arbitrarily close
to~1 (see \cite{KLP11}), and
thus there are integers
$n$ and sets of residues modulo~$n$ consisting of 99\% of all residues, with the set of pairwise
products lying in the remaining 1\% of the residues.  However, it is not easy to find a numerical
example that beats 50\%.   In \cite{KLP11}, an example of a number $n$ with
about $1.61\times10^8$ decimal digits was given with $D(n)>\frac12$; it is not known if there
are any substantially smaller examples, say with fewer than $10^8$ decimal digits.

In \cite{KLP2} the maximal order of $D(n)$ was essentially found:
There are positive constants $c,C$ such that for all sufficiently large $n$, we have
$$
D(n)\le1-\frac{c}{(\log\log n)^{1-\frac\e2\log2}(\log\log\log n)^{1/2}}
$$
and there are infinitely many $n$ with
$$
D(n)\ge1-\frac{C}{(\log\log n)^{1-\frac\e2\log2}(\log\log\log n)^{1/2}}.
$$

In this paper we consider two related questions.  First, if $\A$ is a set of integers which
is both sum-free and product-free, how large may the upper density of $\A$ be?
Second, set 
\[
D'(n) := \max\{ D(\A): \A~\mbox{is a sum-free, product-free subset of}~~ \Z/n\Z\}.
\]
What is the maximal order of $D'(n)$?  We prove the following results.
\begin{theorem}
\label{prop-upperdens}
If $\A$ is a set of positive integers that is both product-free and sum-free,
then $\A$ has upper density at most $\frac12\left(1-\frac1{5a_0}\right)$, where $a_0$ is the least element of $\A$.
\end{theorem}

\begin{theorem}
\label{prop-upper}
There is a positive constant $\kappa$ such that for all sufficiently large numbers $n$, 
$$
D'(n)\le\frac12-\frac{\kappa}{(\log\log n)^{1-\frac\e2\log2}(\log\log\log n)^{1/2}}.
$$
\end{theorem}

\begin{theorem}
\label{prop-lower}
There is a positive constant $\kappa'$ and infinitely many integers $n$ with
$$
D'(n)\ge\frac12-\frac{\kappa'}{(\log\log n)^{1-\frac\e2\log2}(\log\log\log n)^{1/2}}.
$$
\end{theorem}

Note that $D'(5)=\frac25$ and if $5|n$, then $D'(n)\ge\frac25$.
A possibly interesting computational problem is to numerically
exhibit some $n$ with $D'(n)>\frac25$.  Theorem~\ref{prop-lower}
assures us that such numbers exist, but the least example might be
very large.

One might also ask for the densest possible set $\A$ for which $\A$, $\A+\A$, and $\A\cdot\A$
are pairwise disjoint.  However, Proposition~\ref{prop-spf} below implies immediately that any
sum-free, product-free set $\A\subset\Z/n\Z$ with $D(\A)>\frac25$ also has $\A+\A$ and
$\A\cdot\A$ disjoint.  Thus, from Theorem~\ref{prop-lower},
we may have these three sets pairwise disjoint with $D(\A)$ arbitrarily
close to~$\frac{1}{2}$.
\section{The upper density}

Here we prove Theorem \ref{prop-upperdens}.
We begin with some notation that we use in this section:
For a set of positive integers $\A$,
we write $\A(x)$ for $\A\cap[1,x]$.
If $a$ is an integer, we write $a+\A$ for $\{a\}+\A$.

\begin{lemma}
\label{lem-interval}
Suppose that $\A$ is a sum-free set of positive integers and that $I$ is an
interval of length $y$ in the positive reals.  Then
$$
|\A\cap I|\le \frac y2+O_{\A}(1).
$$
\end{lemma}
\begin{proof}
Let $N=|\A\cap I|$.  For any positive integer $m$ we have $|(m+\A)\cap I|\ge |\A\cap I|-m$.  Let $a_0$ be the least element of $\A$.
Thus $|(a_0+\A)\cap I|\ge N-a_0$.  But $a_0+\A$ is disjoint from $\A$, so 
$$
2N-a_0\le|A\cap I|+|(a_0+\A)\cap I|\le|I\cap\Z|=y+O(1).
$$
Solving this inequality for $N$ proves the result.
\end{proof}

For a set $\A$ of positive integers and a real number $x>0$, let 
$$
\delta_x := 1- 2\frac{|\A(x)|}{x},~\hbox{ so that }~ |\A(x)|=\frac12(1-\delta_x)x.
$$
Note that $\delta_x\ge0$ for $|\A(x)| \le \frac{1}{2}.$
\begin{lemma}
\label{lem-a1a2}
Suppose that $\A$ is a sum-free set of positive integers and that
$a_1,a_2\in\A$.  Then for all $x>0$, 
$$
|(a_1+\A(x-a_1))\cap(a_2+\A(x-a_2))|\ge\frac12(1-3\delta_x)-(a_1+a_2).
$$
\end{lemma}
\begin{proof}
We have the sets $\A(x),a_1+\A(x-a_1),a_2+\A(x-a_2)$ all lying in
$[1,x]$ and the latter two sets are disjoint from the first set
(since $\A$ is sum-free).  Thus,
\begin{align*}
|(a_1&+\A(x-a_1))\cap(a_2+\A(x-a_2))|\\
&=|a_1+\A(x-a_1)|+|a_2+\A(x-a_2)|-|(a_1+\A(x-a_1))\cup(a_2+\A(x-a_2))|\\
&\ge |a_1+\A(x-a_1)|+|a_2+\A(x-a_2)|-(x-|\A(x)|)\\
&\ge (|\A(x)|-a_1) +( |A(x)| - a_2)+ (|A(x)|-x) = ~3 |A(x)| - x - (a_1+a_2).
\end{align*}
But $3|\A(x)|-x=\frac12(1-3\delta_x)x$, so this completes the proof.
\end{proof}

For a set $\A$ of positive integers, define the {\em difference set}
$$
\Delta\A :=\{a_1-a_2:a_1,a_2\in\A\}.
$$
Further, for an integer $g$, let 
$$
\A_g :=   \A\cap(-g+\A)     =\{a\in\A:a+g\in\A\}.
$$

\begin{cor}
\label{cor-shiftest}
If $\A$ is a sum-free set of positive integers and $g\in\Delta\A$
then, for any $x>0$, 
$$
|\A_g(x)|\ge\frac{1}{2}(1-3\delta_x)\,x+O(1),
$$
in which  the implied constant
depends on both $g$ and $\A$.
\end{cor}
\begin{proof}
Suppose that $g\in\Delta\A$, so that there exist  $a_1,a_2\in\A$  such that
$a_1-a_2=g$.  If $a\in\A(x-a_1)$ and $a+a_1\in a_2+\A(x-a_2)$,
then $a+g=a+a_1-a_2\in\A$, so that $a\in\A_g$.  That is,
$\A_g(x-a_1)$ contains $-a_1+(a_1+\A(x-a_1))\cap(a_2+\A(x-a_2))$.  
Thus, by Lemma~\ref{lem-a1a2},
$$
|\A_g(x-a_1)|\ge|(a_1+\A(x-a_1))\cap(a_2+\A(x-a_2))|\ge\frac{1}{2}(1-3\delta_x)-(a_1+a_2),
$$
from which the corollary follows.
\end{proof}
\begin{cor}
\label{cor-twofifths}
If $\A$ is a sum-free set of positive integers with upper density
greater than $\frac{2}{5}$, then $\Delta\A$ is a subgroup of $\Z$.
\end{cor}
\begin{proof}
Since $\Delta\A$ is closed under multiplication by $-1$, it
suffices to show that if $g_1,g_2\in\Delta\A$, then $g_1+g_2\in\Delta\A$.  
If $g_1+\A_{g_1}$ contains a member $a$ of $\A_{g_2}$,
then $a-g_1\in\A$ and $a+g_2\in\A$, so that $g_1+g_2\in\Delta\A$.
Note that $g_1+\A_{g_1}$ and $\A_{g_2}$ are both subsets of $\A$.
Now by Corollary~\ref{cor-shiftest}, if $g_1+\A_{g_1}$ and $\A_{g_2}$
were disjoint, we would have for each positive real number $x$,
$$
(1-3\delta_x)x+O(1)\le\frac12(1-\delta_x)x,
$$
so that $\delta_x\ge\frac15+O(\frac1x)$.   Hence $\liminf\delta_x\ge\frac15$, contradicting the
assumption that $\A$ has upper density greater than $\frac25$.  Thus, $g_1+\A_{g_1}$ and $\A_{g_2}$
are not disjoint, which as we have seen, implies that $g_1+g_2\in\Delta\A$.  This completes the proof.
\end{proof} 

\begin{remark}
Corollary \ref{cor-twofifths} is best possible, as can be seen by taking $\A$ as the set 
of positive integers that are either~2 or 3~modulo~5.
\end{remark}

We now prove the following result which immediately implies Theorem~\ref{prop-upperdens}.
\begin{prop}
\label{prop-upperdensbetter}
Suppose that $\A$ is a sum-free set of positive integers with least member $a_0$.
Suppose in addition that $\{a_0\}\cdot\A$ is disjoint from $\A$.  Then the upper density of $\A$
is at most $\frac12\left(1-\frac1{5a_0}\right)$.
\end{prop}
\begin{proof}
If the upper density of $\A$ is at most $\frac25$, the result holds trivially, so we may
assume the upper density exceeds $\frac25$.  It follows from Corollary~\ref{cor-twofifths}
that $\Delta\A$ is the set of multiples of some positive number $g$, which is necessarily
either~1 or~2.  (If $g\ge3$, then the upper density of $\A$ would be at most $\frac13$.)

Suppose that $g=2$ so that $\Delta\A$ consists of all even numbers.  Then either $\A$ consists
of all even numbers or all odd numbers.  In the former case, the set $\{\frac12\}\cdot\A$
is a sum-free set of positive integers, and so has upper density at most~$\frac12$.  It follows
that $\A$ has upper density at most~$\frac14$, a contradicition.

Now suppose that $\A$ consists solely of odd numbers.  For any real number $x\ge a_0$, 
both $\{a_0\}\cdot\A(x/a_0)$ and $\A(x)$ consist only of odd numbers, they are disjoint,
and they lie in $[1,x]$.  Thus, by Lemma~\ref{lem-interval},
$$
|\A(x)|\le\frac12x-\left|\A\left(\frac x{a_0}\right)\right|+O(1)~\hbox{ and }~
|\A(x)|\le\frac12\left(x-\frac x{a_0}\right)+\left|\A\left(\frac x{a_0}\right)\right|+O(1).
$$
Adding these two inequalities and dividing by 2 gives that $|\A(x)|\le(\frac12-\frac1{2a_0})x+O(1)$,
so that $\A$ has upper density at most $\frac12-\frac1{2a_0}$, giving the result in this case. 

It remains to consider the case that $g=1$, that is, $\Delta\A=\Z$.
Let $x\ge a_0$ be any real number and 
consider the two sets $a_0+\A(x-a_0)$ and $\{a_0\}\cdot\A_{-1}(x/a_0)$.  They both lie in $[1,x]$ and
by hypothesis  are both disjoint from $\A(x)$.  If these two sets share an element in common
then there would be some $a\in\A_{-1}(x/a_0)$ with $a_0a-a_0\in\A(x-a_0)$.  
In this case $a_0(a-1)\in\A$ and
also $a_0\in\A$ and $a-1\in\A$, which contradicts our hypothesis.  We conclude that  the three sets
$a_0+\A(x-a_0)$, $\{a_0\}\cdot\A_{-1}(x/a_0)$, and $\A(x)$ must be pairwise disjoint.  Thus,
$$
\left|\A_{-1}\left(\frac{x}{a_0}\right)\right|
=\left|\{a_0\}\cdot\A_{-1}\left(\frac{x}{a_0}\right)\right|
\le x-|\A(x)|-|a_0+\A(x-a_0)|=\delta_xx+O(1).
$$
On the other hand, using Lemma~\ref{lem-interval} and Corollary~\ref{cor-shiftest}, 
\begin{align*}
\left|\A_{-1}\left(\frac x{a_0}\right)\right|&\ge|\A_{-1}(x)|-\frac12\left(1-\frac1{a_0}\right)x+O(1)\\
&\ge\frac12(1-3\delta_x)x-\frac12\left(1-\frac1{a_0}\right)x+O(1)
=\left(\frac1{2a_0}-\frac32\delta_x\right)x+O(1).
\end{align*}
Putting these two inequalities together and dividing by $x$, we obtain
$$
\delta_x\ge\frac1{2a_0}-\frac32\delta_x+O\left(\frac1x\right).
$$
Hence $\delta_x\ge\frac1{5a_0}+O(\frac1x)$, implying that
$\liminf\delta_x\ge\frac1{5a_0}$, whence
$$
\bar{d}(\A) = \limsup_{x \to \infty} \frac{1}{x} |\A(x)|
=\limsup_{x \to \infty} \frac{1}{2}(1-\delta_x)
= \frac{1}{2} -  \liminf_{x \to\infty} \frac{1}{2}\delta_x \ge \frac{1}{2}(1- \frac{1}{5a_0}),
$$
which proves the proposition.
\end{proof}

\section{An upper bound for the density in $\Z/n\Z$}

In this section we prove Theorem \ref{prop-upper}.
We use the following  theorem of Kneser~\cite{K}; see also \cite[Theorem~5.5]{TV}.
\begin{theorem}[Kneser]
\label{th-Kneser}
Suppose in an abelian group $G$ (written additively) we have finite nonempty sets $\A,\B,\C$
where $\A+\B=\C$.  Let $H$ be the stabilizer of $\C$ in $G$, that is, $H$ is the subgroup
of elements $g\in G$ with $g+\C=\C$.  Then
$$
|\C|\ge|\A+H|+|\B+H|-|H|\ge|\A|+|\B|-|H|.
$$
\end{theorem}

We next deduce restrictions on the structure of sum-free sets having density greater than~$\frac{2}{5}.$
\begin{prop}
\label{prop-spf}
Suppose that $n$ is a positive integer and $\A\subset\Z/n\Z$ is sum-free.
If $D(\A)>\frac{2}{5}$, then $n$ is even and $\A$ is a subset of the odd residues in $\Z/n\Z$.
\end{prop}
\begin{proof}
It follows from \cite{DY69}  that $D(\A) > \frac{2}{5}$ implies $n$ must be even (see the comments in
Section~\ref{sec:introduction}).   
The result  holds for $n=2$ since a sum-free set cannot contain $0$.
It also holds for $n=4$ since the double of an odd residue is $2$, so
the only option for $\A$ is $\{1,3\}$.
We now suppose $n \ge 6$  is even and proceed by induction 
assuming that the proposition holds for all even numbers smaller than $n$.
Let $\C$ denote the set
of residues mod~$n$ of the form $a+b$, where $a,b\in\A$.  Since $|\A|>\frac25n$
and $\A$ is sum-free, we have $|\C|<\frac35n$, so that $|\C|<2|\A|-1$ (using $n\ge6$).  We apply
Kneser's theorem in the group $G=\Z/n\Z$ to conclude that the stabilizer $H$ of $\C$ must be nontrivial.
Thus $H=\langle h\rangle$,
where $h|n$ and $h<n$.  

Next note  that if $\psi$ denotes the projection map of $\Z/n\Z$ to $\Z/h\Z$,
then $\psi(\A)$ is still sum-free.  To show this, suppose not, whence there are  $a_1,a_2,a_3\in\A$ and $\psi(a_1)
+\psi(a_2)=\psi(a_3)$. Now  $a_3\equiv a_1+a_2\pmod h$, so that there is some $c\in\C$ (namely $c=a_1+a_2$)
with $a_3\in H+c$.  But $H+c\subset H+\C=\C$, so $a_3=a_1'+a_2'$ for some $a_1',a_2'\in\A$, contradicting
the assumption that $\A$ is sum-free.  

The projection map $\psi$ is $n/h$ to $1$,
so $|\psi(\A)|\ge |\A|/(n/h)>\frac{2}{5}h$, and this  implies that $h$ must be even.
By the induction hypothesis $\psi(\A)$ cannot contain any even residues
modulo~$h$.  But even residues in $\A$ reduce to even residues modulo~$h$, so $\A$ cannot
contain any even residues modulo $n$.  This completes the proof.
\end{proof}

We now prove Theorem \ref{prop-upper}.
For those $n$ with $D'(n)\le\frac25$, the result holds for any number $\kappa$, so assume 
that $D'(n)>\frac25$.
Let $\A\subset\Z/n\Z$ be a product-free, sum-free set with $D(\A)=D'(n)$.  
By Proposition~\ref{prop-spf}, we have that $n$ is even and that
$\A$ is a subset of the odd residues modulo~$n$.  
Suppose that $k$ is an integer with $n\le 2^k<2n$.  Let $N=2^{2k}n$ and let $\B$ be the set of positive integers of the
form $2^jb$ where $j\le k$ and $b\le N/2^j=2^{2k-j}n$, such there is some $a\in\A$ with
$b\equiv a\pmod n$.  Then the members of $\B$ are in $[1,N]$ and
\begin{equation}
\label{lowerB}
|\B|=\sum_{j=0}^k2^{2k-j}|\A|=2^k\left(2^{k+1}-1\right)|\A|>\left(1-\frac1n\right)2^{2k+1}|\A|.
\end{equation}

We note that $\B$ is product-free as a set of residues modulo~$N$.
Indeed, suppose $2^{j_i}b_i \in\B$, for $i=1,2,3$ and 
$$
2^{j_1}b_1 2^{j_2}b_2\equiv 2^{j_3}b_3\pmod{N}.
$$
Let $a_i\in\A$ be such that $b_i\equiv a_i\pmod n$ for $i=1,2,3$.  We have that $a_1,a_2,a_3$ are
odd, and since $n$ is even, this implies that $b_1,b_2,b_3$ are odd.
Using $j_1+j_2\le 2k, j_3\le k$ and $2^{2k}|N$, we have
$j_1+j_2=j_3$.  Hence $a_1a_2\equiv a_3\pmod{n}$, a violation of the assumption that
$\A$ is product-free modulo~$n$.  We conclude that $\B$ is product-free modulo~$N$.

It now follows from Theorem~1.1 in~\cite{KLP2} that for $n$ sufficiently large,
$$
|\B|\le N\left(1-\frac{c}{(\log\log N)^{1-\frac{\rm e}2\log2}(\log\log\log N)^{1/2}}\right).
$$
Further, since $N$ is of order of magnitude $n^3$, we have that $\log\log N=\log\log n+O(1)$,
and so for any fixed choice of $c_0<c$ we have for $n$ sufficiently large that
$$
|\B|\le N\left(1-\frac{c_0}{(\log\log n)^{1-\frac{\rm e}2\log2}(\log\log\log n)^{1/2}}\right).
$$
Thus, from our lower bound for $|\B|$ in \eqref{lowerB} we have
$$
|\A|<\frac{N}{2^{2k+1}}\left(1-\frac1n\right)^{-1}
\left(1-\frac{c_0}{(\log\log n)^{1-\frac{\rm e}2\log2}(\log\log\log n)^{1/2}}\right).
$$
Since $N/2^{2k+1}=n/2$, it follows that for any fixed $c_1<c_0$ and $n$ sufficiently large,
we have
$$
|\A|<\frac n2
\left(1-\frac{c_1}{(\log\log n)^{1-\frac{\rm e}2\log2}(\log\log\log n)^{1/2}}\right).
$$
We thus may choose $\kappa$ as any number smaller than $c/2$.  This concludes the proof
of Theorem~\ref{prop-upper}.

\section{Examples with large density}

In this section we prove Theorem~\ref{prop-lower}.
We follow the argument in \cite{KLP11} with a supplementary estimate from~\cite{KLP2}.
Let $x$ be a large number, let $\ell_x$ be the least common multiple of the integers in $[1,x]$
and let $n_x=\ell_x^2$.  Then $n_x=\e^{(2+o(1))x}$ as $x\to\infty$ so that $\log\log n_x=\log x+O(1)$.
For a positive integer $m$, let $\Omega(m)$ denote the number of
prime factors of $m$ counted with multiplicity.
Let $k=k(x)=\lfloor\frac\e4\log\log n_x\rfloor$, let
$$
\D'_x=\left\{d|\ell_x:d~{\rm odd},~k<\Omega(d)<2k\right\},
$$
and let $\A$ be the set of residues $a$ modulo~$n_x$ with $\gcd(a,n_x)\in \D'_x$.  Then
$\A$ is product-free (cf.~Lemma~2.3 in~\cite{KLP11}), and since $n_x$ is even and every 
residue in $\A$ is odd, we have that $\A$ is sum-free as well.  We shall now establish a sufficiently
large lower bound on $D(\A)$ to show that $D'(n_x)$ satisfies the inequality in the theorem
with $n=n_x$.
 
For $d\in\D'_x$, the number of $a\pmod{n_x}$ with $\gcd(a,n_x)=d$ is $\varphi(n_x)/d$, so that
\begin{equation}
\label{DofA}
D(\A)=\frac{\varphi(n_x)}{n_x}\sum_{d\in\D'_x}\frac1d
=\frac{\varphi(n_x)}{n_x}\left(\sum_{\substack{d|\ell_x\\d~{\rm odd}}}\frac1d
-\sum_{\substack{d|\ell_x\\d~{\rm odd}\\d\not\in\D'_x}}\frac1d\right)
\ge\frac{\varphi(n_x)}{n_x}\left(\sum_{\substack{d|\ell_x\\d~{\rm odd}}}\frac1d
-\sum_{\substack{d|\ell_x\\d\not\in\D'_x}}\frac1d\right).
\end{equation}
We have
$$
\sum_{\substack{d|\ell_x\\d~{\rm odd}}}\frac1d
=\prod_{2<p\le x}\frac{p}{p-1}\cdot\prod_{\substack{2<p\le x\\p^a\|\ell_x}}\left(1-\frac1{p^{a+1}}\right)
\ge\prod_{2<p\le x}\frac{p}{p-1}\cdot\left(1-\frac1x\right)^{\pi(x)}
$$
and, since $\varphi(n_x)/n_x=2^{-1}\cdot\prod_{p|n_x,\, p>2} (1-1/p)$, we find that
\begin{equation}
\label{1stterm}
\frac{\varphi(n_x)}{n_x}
\sum_{\substack{d|\ell_x\\d~{\rm odd}}}\frac1d
\ge\frac12\left(1-\frac1x\right)^{\pi(x)}\ge\frac12-\frac{\pi(x)}x.
\end{equation}

We now use (6.2) in \cite{KLP2} which is the assertion that
$$
\sum_{\substack{P(d)\le x\\\Omega(d)\not\in(k,2k)}}\frac1d
\ll \frac{(\log x)^{\frac\e2\log2}}{(\log\log x)^{1/2}}.
$$
Here, $P(d)$ denotes the largest prime factor of $d$.  Since this sum includes every integer $d|\ell_x$
that is not in $\D'_x$, we have
$$
\frac{\varphi(n_x)}{n_x}\sum_{\substack{d|\ell_x\\d\not\in\D'_x}}\frac1d
\ll \frac{\varphi(n_x)}{n_x}\cdot\frac{(\log x)^{\frac\e2\log2}}{(\log\log x)^{1/2}}
\ll \frac1{(\log x)^{1-\frac\e2\log2}(\log\log x)^{1/2}},
$$
where we use Mertens' theorem in the form
$\varphi(n_x)/n_x=\prod_{p\le x}(1-1/p)\ll1/\log x$ for the last step.
Putting this estimate and \eqref{1stterm} into \eqref{DofA}, we get
$$
D(\A)\ge \frac12-\frac{\pi(x)}x-\frac{c'}{ (\log x)^{1-\frac\e2\log2}(\log\log x)^{1/2}}
$$
for some positive constant $c'$.  Using $\pi(x)/x\ll1/\log x$ and $\log x=\log\log n_x+O(1)$,
we have
$$
D(\A)\ge\frac12-\frac{\kappa'}{(\log\log n_x)^{1-\frac\e2\log2}(\log\log\log n_x)^{1/2}}
$$
for any fixed constant $\kappa'>c'$ and $x$ sufficiently large.  Thus, $D'(n_x)$ satisfies
the condition of Theorem~\ref{prop-lower} for $x$ sufficiently large, completing the proof.

\section*{Acknowledgments}
We thank Albert Bush, Chris Pryby, and Joseph Vandehey for raising the question of sets which are both
sum-free and product-free.
PK was supported in part by grants from
the G\"oran Gustafsson Foundation, and the Swedish Research Council.
JCL was supported in
part by NSF grant  DMS-1101373.  CP was supported in
part by NSF grant DMS-1001180. 


\end{document}